\documentclass[numbook,envcountreset,envcountsame]{svjour3}
\usepackage{amsmath,amssymb}
\usepackage{xcolor}
\usepackage{url}
\usepackage{bbm}
\usepackage[greek,english]{babel}
\usepackage{hyperref}
\usepackage{natbib}
\hypersetup{hypertex=true,
            colorlinks=true,
            linkcolor=red,
            anchorcolor=black,
            citecolor=green,
            urlcolor=blue}

\newtheorem{Assumption}[theorem]{Assumption}

\usepackage{geometry}
\begin{document}

\title{Parameter estimation for Vasicek model driven by a general Gaussian noise}
\author{Xingzhi Pei}
\institute{X. Pei \at
            School of Mathematics and Statistics, Jiangxi Normal University, 330022 Nanchang, PR China\\
            \email{\href{mailto:SachikoXpei@jxnu.edu.cn}{SachikoXpei@jxnu.edu.cn}}
}
\date{Received: 1900.01.01/ Accepted: 1900.01.01}
\maketitle
\begin{abstract}
 This paper developed an inference problem for Vasicek model driven by a general Gaussian process. We construct a least squares estimator and a moment estimator for the drift parameters of the Vasicek model, and we prove the consistency and the asymptotic normality. Our approach extended the result of Xiao and Yu (2018) for the case when noise is a fractional Brownian motion with Hurst parameter $H\in[1/2,1)$.
\end{abstract}

\section{Introduction}
We are interested in the statistical interference for the Vasicek model defined by the following stochastic differential equation(SDE)
\begin{align}\label{model}
dX_{t}=k(\mu -X_{t})dt+\sigma dG_{t},t\in \lbrack 0,T],T\geq 0,x_{0}=0,
\end{align}
where $(G_{t})_{t\geq 0}$ is a general one-dimensional centered Gaussian process. We noticed that the volatility parameter $\sigma>0$ can be estimated by the power variation method. Without loss of generality, we assume that $\sigma=1$. Assuming that there is only one trajectory $(X_{t},t\geq 0)$, we construct a least squares estimator and a moment estimator, and study its asymptotic behavior.\par
The Vasicek model of \cite{Vasicek1977} has a wide range of applications in many fields, such as economics, finance, biology, medical and environmental sciences. In the economic field, it has been used to describe the fluctuation of interest rates, please refer to \cite{Jingzhi2012}. In the financial field,  it can also be used as a random investment model in {\cite{San2020}}.\par
If the parameter in the drift function of model is unknown, an important problem is to estimate the drift coefficient based on the observation. When the noise is Brownian motion, the statistical inference for Vasicek process are well studied in the literature, e.g. a maximum likelihood method was proposed in \cite{Fergusson2015}, whereas \cite{yang2013} studied a least squares approach.\par
The research methods are different when the drift parameter $k>0$ or $k<0$. When the Brownian motion in the vasicek model was replaced by the fractional Brownian motion \cite{Xiao2017}, with the Hurst parameter greater than or equal to one-half, the asymptotic theory for k was proved, the stationary case for $k>0$, the explosive case for $k<0$, and the null recurrent case for $k=0$, respectively. In these cases, the least squares method is considered, and when $k>0$, the moment estimation method of \cite{Hu2010} is also considered.\par
Based on this, \cite{Xiao2018} extended their work to Vasicek-type models driven by sub-fBm. For the case of non-ergodic and null recurrent, the least squares method was studied. In addition, it can be extended to a more general self-similar process, such as Hermite process, (see \cite{Nourdin2018}). \par
Moreover, when the Brownian motion is replaced by a Gaussian process with self-similarity \cite{Yu2018}, based on some conditions of G, the least squares method was studied and its asymptotic behavior was completed with non-ergodic case $k<0$.\par
In this paper, we consider the Vasicek model driven by a general Gaussian process that fails to be self-similar or have stationary increments, when the persistence parameter k is positive.\par
This paper refers to \cite{Yong2020} and makes the following assumptions about the second-order partial derivative form of the covariance function of a general Gaussian process.
\begin{Assumption}
For $\beta \in (\frac{1}{2},1)$, the covariance function $R(t,s)=E[G_{t}G_{s}]$  for any  $t\neq s\in \lbrack 0,\infty )$
\begin{align}
    \frac{\partial ^{2}}{\partial t\partial s}R(t,s)&=C_{\beta }|t-s|^{2\beta-2}+\Psi (t,s),\end{align}
with  \begin{align}
    |\Psi (t,s)|&\leq C_{\beta }^{^{\prime }}|ts|^{\beta -1},
\end{align}
where the constants $\beta ,C_{\beta }>0,C_{\beta }^{^{\prime }}\geq 0$ do not depend on T. Moreover, for any $t\geq 0$, $R(0,t)=0$.
\end{Assumption}
 We can see that fractional Brownian motion and some other Gaussian processes satisfy  Assumption1.1. From this assumption, we obtain the result as follow.
 When $k>0$, the estmator of $\mu$ is continous-time sample mean, (see \cite{Hu2010}).
\begin{equation}\label{4}
\widehat{\mu }=\frac{1}{T}\int_{0}^{T}X_{t}dt.
\end{equation}
Moreover, following \cite{Xiao2017}, when $k>0$,
the second moment estimator is given by
\begin{equation}
\widehat{k}=\left[ \frac{\frac{1}{T}\int_{0}^{T}X_{t}^{2}dt-(\frac{1}{T}%
\int_{0}^{T}X_{t}dt)^{2}}{C_{\beta }\Gamma (2\beta -1)}\right] ^{-\frac{1}{%
2\beta }}.
\end{equation}
 The LSE is motivated by the argument of minimize a quadratic function of $k$ and $\mu$, respectively
\begin{equation}
L(k,\mu )=\int_{0}^{T}(\overset{\cdot }{X_{t}}-k(\mu -X_{t}))^{2}dt.
\end{equation}
Solving the equations, we can obtain the LSE of $k$ and $\mu$, denoted by $\widehat{k}_{LS}$ and $\widehat{\mu }_{LS}$, respectively.
\begin{equation}\label{109}
\widehat{k}_{LS}=\frac{X_{T}\int_{0}^{T}X_{t}dt-T\int_{0}^{T}X_{t}dX_{t}}{%
T\int_{0}^{T}X_{t}^{2}dt-(\int_{0}^{T}X_{t}dt)^{2}},
\end{equation}
\begin{equation}
\widehat{\mu }_{LS}=\frac{X_{T}\int_{0}^{T}X_{t}^{2}dt-\int_{0}^{T}X_{t}dX_{t}%
\int_{0}^{T}X_{t}dt}{X_{T}\int_{0}^{T}X_{t}dt-T\int_{0}^{T}X_{t}dX_{t}},
\end{equation}
where  the integral $\int_{0}^{T}X_{t}dX_{t}$ can be interpret as an It$\hat{o}$-Skorohod integral (\cite{Xiao2017}).\par
In this paper, we will prove the strong consistency and the central limit theorems for the four estimators, these results are stated in the following theorems.

\begin{theorem}\label{theorem1}
When assumption1.1 is satisfied, both the least squares estimator and the  moment estimator of $\mu$ and $k$ are strongly consistent, i.e
\begin{equation}
\underset{T\rightarrow \infty }{\lim }\widehat{\mu }=\mu ,      \underset{%
T\rightarrow \infty }{\lim }\widehat{\mu }_{LS}=\mu ,      a.s..
\end{equation}
\begin{equation}
\underset{T\rightarrow \infty }{\lim }\widehat{k}=k,        \underset{T\rightarrow
\infty }{\lim }\widehat{k}_{LS}=k,     a.s..
\end{equation}
\end{theorem}
\begin{theorem}\label{theorem2}
Assume assumption 1.1 is satisfied. When $\beta\in(1/2,1)$, both $T^{1-\beta }(\widehat{\mu }-\mu )$ and $T^{1-\beta }(\widehat{\mu }_{LS}-\mu )$ are asymptotically normal as $T\rightarrow \infty$, namely,
\begin{equation}
T^{1-\beta }(\widehat{\mu }-\mu )\overset{law}{\rightarrow }N(0,\frac{1}{%
k^{2}}),\\
T^{1-\beta }(\widehat{\mu }_{LS}-\mu )\overset{law}{\rightarrow }N(0,\frac{1%
}{k^{2}}),\\
\end{equation}
when $\beta \in (\frac{1}{2},\frac{3}{4})$, both $\sqrt{T}(\widehat{k}-k)$ and $\sqrt{T}(\widehat{k}_{LS}-k)$ are asymptotically normal as $T\rightarrow \infty$, namely,
\begin{equation}
\sqrt{T}(\widehat{k}_{LS}-k)\overset{law}{\rightarrow }N(0,4ka^{2}\sigma _{\beta }^{2}),\\
\sqrt{T}(\widehat{k}-k)\overset{law}{\rightarrow }N(0,\sigma _{\beta}^{2}k/4\beta ^{2})
\end{equation}
 where $a=C_{\beta }\Gamma (2\beta -1)k^{-2\beta }$, $\sigma _{\beta }^{2}=(4\beta -1)[1+\frac{\Gamma (3-4\beta )\Gamma (4\beta-1)}{\Gamma (2\beta )\Gamma (2-2\beta )}]$.
\end{theorem}
The outline of the paper is the following. First, we provide some basic elements of stochastic calculus with respect to the Gaussian process which are helpful for some of the arguments we use and some of the technical results used in various proofs. In Sect.3 and 4 we derive our estimator, prove consistency and  asymptotic normality respectively.\par
\section{Preliminary}
In this section, we describe some basic facts on stochastic calculus with respect to the Gaussian process and  recall the main results in \cite{Nualart2005} concerning the central limit theorem for multiple integrals, for more complete presentation on the subject can be find in \cite{Yong2020}.\par
Defined on a complete probability space$(\Omega,\mathcal{F},P)$, the $\mathcal{F}$ is generated by the Gaussian family $G$. Denote $G={G_{t},t\in[0,T]}$ as a continuous centered Gaussian process, and suppose in addition that the covariance function $R$ is continuous.
\begin{equation}
\mathbb{E}(G_{t}G_{s})=R(s,t),s,t\in[0,T],
\end{equation}
let $\varepsilon$ denote the space of all real valued step functions on [0,T]. The Hilbert space $\mathfrak{H}$ is defined as the closure of $\varepsilon$ endowed with the inner product.
\begin{equation}
\langle\mathbbm{1}_{[a,b)},\mathbbm{1}_{[c,d)}\rangle_{\mathfrak{H}}=\mathbb{E}((G_{b}-G_{a})(G_{d}-G_{c})).
\end{equation}
If $G={G_{h}, h\in{\mathfrak{H}}}$ as the isonormal Gaussian process on the probability space, indexed by the elements in the Hilbert space $\mathfrak{H}$, $G$ is a Gaussian family of random variables as follows
\begin{equation}
\mathbb{E}(G_{g}G_{h})=\langle g,h>_{\mathfrak{H}},\forall g,h \in\mathfrak{H}.
\end{equation}
The following proposition is an extension of Theorem 2.3 of \cite{Jolis2007}, which gives the inner products representation of the Hilbert space $\mathfrak{H}$ and the References therein.
\begin{proposition}
Denote $\mathcal{V}_{[0,T]}$ as the set of bounded variation functions on $[0,T]$, then $\mathcal{V}_{[0,T]}$ in dense in $\mathcal{V}$ and we have
\begin{equation}
  \langle f,g\rangle_\mathfrak{H}=\int_{[0,T]^{2}}R(t,s){v}_{f}(dt){v}_{g}(ds),      \forall f,g\in \mathcal{V}_{[0,T]},
\end{equation}
where ${v}_{g}$ is the Lebesgue-Stieljes signed measure associated with $g^{0}$ defined as
\begin{equation}
g^{0}= \left\{
\begin{array}{rcl}
g(x),&& {if \quad x\in [0,T];}\\
0,&& {otherwise.}\\
\end{array}
\right.
\end{equation}
Furthermore, if covariance function $R(t,s)$ satisfies Assumption1.1, then
\begin{equation}
\langle f,g\rangle_\mathfrak{H}=\int_{[0,T]^{2}}R(t,s)\frac{\partial ^{2}R(t,s)}{\partial
t\partial s}dtds,                  \forall f,g\in \mathcal{V}_{[0,T]}.
\end{equation}
\end{proposition}

\begin{corollary}
 If Assumption1.1 is satisfied, there exists a constant $C>0$ independent of T, such that for all $s,t\geq 0$,
 \begin{equation}
 \mathbb{E}(G_{t}-G_{s})^{2}\leq C_{\beta}|t-s|^{2\beta},
 \end{equation}
and when $s=0$, we have $E(G_{t}^{2})\leq C_{\beta}^{'}t^{2\beta}$.
\end{corollary}

\begin{proof}
 \begin{equation}
 \begin{aligned}
 \mathbb{E}(G_{t}-G_{s})^{2}&=\int_{[s,t]^{2}}\frac{\partial ^{2}R(u,v)}{\partial u\partial v}dudv\\
 &\leq \int_{[s,t]^{2}}|u-v|^{2\beta-2}dudv+\int_{[s,t]^{2}}|uv|^{\beta-1}dudv
 &\leq C_{\beta}|t-s|^{2\beta}.
 \end{aligned}
 \end{equation}
 Hence, we deduce the desired result.
\end{proof}

\begin{remark}
Denote $\mathfrak{H}^{\bigotimes p}$ and $\mathfrak{H}^{\bigodot p}$ as the pth tensor product and the pth symmetric tensor product of the Hilbert space $\mathfrak{H}$. Let $\mathcal H_p$ be the Wiener chaos with respect to G. It is defined as the closed linear subspace of $L^{2}(\Omega)$ generated by the random variables ${H_{p}(G(h))h\in{\mathfrak{H}}}$, where $H_{p}$ is the pth Hermite polynomial defined by
\begin{equation}
H_{p}(x)=\frac{(-1)^{p}}{p!}e^{\frac{x^{2}}{2}}\frac{d^{p}}{dx^{p}}e^{-%
\frac{x^{2}}{2}},p\geq 1,
\end{equation}
and $H_{0}(x)=1$. We have the identity $I_{p}(h^{\bigotimes p})=H_{p}(G(h))$ for any $h\in H$ where $I_{p}(\cdot )$ is the generalized Wiener-It$\widehat{o}$ stochastic integral. Then the map $I_{p}$ provides a linear isometry between $\mathfrak{H}^{\bigodot p}$ and $H_{p}$. Here $H_{o}=R$ and $I_{0}(x)=x$ by the convention.\par
 We choose $e_{k}$ to be a complete orthonormal system in the Hilbert space $\mathfrak{H}$. The q-th contration between $ f\in \mathfrak{H}^{\bigodot m} $ and $ g\in \mathfrak{H}^{\bigodot n}$ is an element in $\mathfrak{H}^{m+n-2q}$ that is defined by
\begin{equation}
f\bigotimes_{q}g=\overset{\infty }{\underset{i_{1},\cdot \cdot \cdot ,i_{q}=1}{\sum
}}\langle f,e_{i_{1}\bigotimes \cdot \cdot \cdot \bigotimes}e_{i_{q}}\rangle_{\mathfrak{H}^{\bigotimes{q}}}\bigotimes \langle g,e_{i_{1}\bigotimes \cdot \cdot\cdot \bigotimes}e_{i_{q}}\rangle_{\mathfrak{H}^{\bigotimes{q}}},for q=1,...,m\wedge n.
\end{equation}
Then we have the following product formula for the multiple integrals,
\begin{equation}
I_{p}(g)I_{q}(h)=\underset{r=0}{\overset{p\wedge q}{\sum }}%
r!\tbinom{p}{r}\tbinom{q}{r}I_{p+q-2r}(g\widetilde{\bigotimes}_{r}h).
\end{equation}
\end{remark}
The following theorem 2.3, known as the fourth moment theorem, provides necessary and sufficient conditions for the asymptotic theory of the persistent parameter$k$, (see \cite{Nualart2005}).
\begin{theorem}
Let $n\geq 2$ be a fixed integer. consider a collection of elements${f_{T}, T>0}$ such that $f_T\in \mathfrak{H}\bigodot n$ for every $T>0$. Assume further that
\begin{equation}
\underset{T\rightarrow \infty }{\lim }\mathbb{E}[I_{n}(f_{T})^{2}]=\underset{%
T\rightarrow \infty }{\lim }n!||f_{T}||_{\mathfrak{H}^{\bigotimes n}}^{2}=\sigma ^{2}.
\end{equation}
then the following conditions are equivalent:\par
(1)$\underset{T\rightarrow \infty }{\lim }E[I_{n}(f_{T})^{4}]=3\sigma ^{4}$.\par
(2)For every q=1,...,n-1,$\underset{T\rightarrow \infty }{\lim }E||f_{T}\bigotimes f_{T}||_{\mathfrak{H}\bigotimes 2(n-q)}=0$.\par
(3)As T tends to infinity, the n-th multiple integrals $\{I_{n}(f_{T}),T\geq 0\}$
converge in distribution to a Gaussian random variable $N(0,\sigma ^{2})$.\par
\end{theorem}

\section{Strong consistency}
\subsection{The moment estimator}
If $k>0$, we can consider estimators of $k$ and $\mu$. The estimators are motivated by \cite{Hu2010} where the stationary and ergodic properties of a process were used to construct a new estimator for $k$ in the fOU model. Then we first consider strong consistency of $\hat\mu$, the solution of the model in (1.1) is given by
\begin{equation}\label{23}
X_{t}=\mu (1-e^{-kt})+\int_{0}^{T}e^{-k(t-s)}dG_{s},
\end{equation}
so the estimator of $\mu$ is the continuous-time sample mean
\begin{equation}\label{24}
\widehat{\mu }=\frac{1}{T}\int_{0}^{T}X_{t}dt.
\end{equation}
Combining (\ref{23}) and (\ref{24}), we can rewrite $\widehat{\mu}$ as,
\begin{equation}\label{25}
\widehat{\mu }=\frac{1}{T}\int_{0}^{T}(1-e^{-kt})\mu dt+\frac{1}{T}\int_{0}^{T}(\int_{0}^{t}e^{-k(t-s)}dG_{t})dt.
\end{equation}
For the second term in (\ref{25}), we first define some important functions that will be used in the proof.
Denote
$F_{T}=\int_{0}^{T}e^{-kt}\int_{0}^{t}e^{ks}dG_{s}dt$, using stochastic Fubini theorem to obtain
\begin{equation}\label{34}
F_{T}=\int_{0\leq s\leq t\leq T}e^{-k(t-s)}dG_{s}dt=\int_{0}^{T}\frac{1}{k}(1-e^{-k(T-s )})dG_{s}
=G_{T}-Z_{T},
\end{equation}
where $Z_{T}=\int_{0}^{T}e^{-k(T-s)}dG_{s}$ is the Wiener-It$\hat{o}$ stochastic integral. And in the remaining part of this paper, C will be a generic positive constant independent of T whose value may differ from line to line.\par
\begin{remark}(see \cite{Yong2020})
For a function $\phi(r)\in\mathcal{V}_{[0,T]}$, we define two norms as
\begin{equation}
\begin{aligned}
\left\Vert\phi\right\Vert_{\mathfrak{H}_{1}}^{2}=C_{\beta}\int_{[0,T]^{2}}\phi(r_{1}\phi(r_{2})|r_{1}-r_{2}|^{2\beta-2}dr_{1}dr_{2},\\
\left\Vert\phi\right\Vert_{\mathfrak{H}_{2}}^{2}=C_{\beta}^{'}\int_{[0,T]^{2}}|\phi(r_{1}\phi(r_{2})|(r_{1}-r_{2})^{\beta-1}dr_{1}dr_{2}.\\
\end{aligned}
\end{equation}
For a funtion $\varphi(r,s)$ in $[0,T]^2$, define an operator from $\mathcal{V}_{[0,T]}^{\otimes {2}}$ to $\mathcal{V}_{[0,T]}$ as follows,
\begin{equation}
(K\varphi)(r)=\int_{0}^{T}|\varphi(r,u)|u^{\beta-1}du.
\end{equation}
\end{remark}
\begin{proposition}(see \cite{Yong2020})
Suppose that Assumption1.1 holds, then for any $\phi(r)\in\mathcal{V}_{[0,T]}$,
\begin{equation}\label{29}
|\left\Vert\phi\right\Vert_{\mathfrak{H}}^{2}-\left\Vert\phi\right\Vert_{\mathfrak{H}_{1}}^{2}|\leq\left\Vert\phi\right\Vert_{\mathfrak{H}_{2}}^{2},
\end{equation}
and for any $\varphi,\psi \in (\mathfrak(V)_{0,T})^{\bigodot 2}$,
\begin{equation}
\begin{aligned}
|\left\Vert\phi\right\Vert_{\mathfrak{H}^{\bigotimes{2}}}^{2}-\left\Vert\phi\right\Vert_{\mathfrak{H}_{1}^{\bigotimes{2}}}^{2}|&\leq\left\Vert\phi\right\Vert_{\mathfrak{H}_{2}^{\bigotimes{2}}}^{2}+2C_{\beta}^{'}\left\Vert K\varphi\right\Vert_{\mathfrak{H}_{1}}^{2},\\
|\langle\varphi,\psi\rangle_{\mathfrak{H}^{\bigotimes{2}}}-\langle\varphi,\psi\rangle_{\mathfrak{H}_{1}^{\bigotimes{2}}}|&\leq|\langle\varphi,\psi\rangle_{\mathfrak{H}_{2}^{\bigotimes{2}}}|+2C_{\beta}^{'}|\langle K\varphi,K\psi\rangle_{\mathfrak{H}_{1}}|.\\
\end{aligned}
\end{equation}
\end{proposition}
The next two propositions are about the asymptotic behaviors of the second moment of $F_{T}$ and the increment $F_{t}-F_{s}$ with $0\leq t,s\leq T$, respectively. First, we need a technical lemma as follows.
\begin{lemma} \label{lemma1}
 Assume $\beta \in (0,1)$, there exists a constant $C>0$ such that for any $T\in \lbrack 0,\infty )$,
\begin{equation}
e^{-kT}\int_{0}^{T}e^{kr}r^{\beta -1}dr\leq C(1\wedge T^{\beta -1}),
\end{equation}
(see lemma 3.3 of \cite{Yong2020}).
\end{lemma}

\begin{proposition}
When $\beta\in (\frac{1}{2},1)$, we can find that
\begin{equation}\label{33}
E(F_{T}^{2})\leq C_{\beta }T^{2\beta }.
\end{equation}
\end{proposition}

\begin{proof}
By It$\hat{o}$ isometry, we have
\begin{equation}
\mathbb{E}[|F_{T}|^2]=\left\Vert f_{T}\right\Vert_{\mathfrak{H}}^2,
\end{equation}
the inequality (\ref{29})implies that
\begin{equation}
|\left\Vert f_{T}\right\Vert_{\mathfrak{H}}^{2}-\left\Vert f_{T}\right\Vert_{\mathfrak{H}_{1}}^{2}|\leq \left\Vert f_{T}\right\Vert_{\mathfrak{H}_{2}}^{2},
\end{equation}
clearly, we have
\begin{equation}
0\leq \int_{0}^{T}(1-e^{-k(T-u)})u^{\beta-1}du \leq \int_{0}^{T}u^{\beta-1}du\leq CT^{\beta},
\end{equation}
so,
\begin{equation}
\left\Vert f_{T}\right\Vert_{\mathfrak{H}_{2}}^{2}=(\frac{1}{k}\int_{0}^{T}(1-e^{-k(T-u)})u^{\beta-1}du)^{2}\leq CT^{2\beta}.
\end{equation}
Meanwhile, we have
\begin{equation}
\begin{aligned}
\left\Vert f_{T}\right\Vert_{\mathfrak{H}_{1}}^{2}&=\frac{1}{k^{2}}\int_{[0,T]^{2}}(1-e^{-k(T-u)})(1-e^{-k(T-v)})|u-v|^{2\beta
-2}dudv\\
&\leq\frac{1}{k^{2}}\int_{[0,T]^{2}}|u-v|^{2\beta-2}dudv+\int_{[0,T]^{2}}e^{-k(T-u)}e^{-k(T-v)}|u-v|^{2\beta-2}dudv\\
&\leq \frac{1}{k^{2}}[\frac{T^{2\beta }}{(2\beta -1)\beta}+ \frac{\Gamma(2\beta -1)}{k^{2\beta }}].\\
\end{aligned}
\end{equation}
Hence, $\left\Vert f_{T}\right\Vert_{\mathfrak{H}}^{2}\leq C_{\beta}^{'}[T^{2\beta}]$, we obtain the desired result in(\ref{33}).
\end{proof}

\begin{proposition}
Assume that assumption1.1 holds, there exists a constant $C>0$ independent of $T$ such that for all $s,t\geq0$,
\begin{equation}
E[|F_{t}-F_{s}|^{2}]\leq C_{\alpha ,\beta }|t-s|^{2\beta }.
\end{equation}
\end{proposition}

\begin{proof}
Firstly, the equality (\ref{34}) implies that
\begin{equation}
E[|F_{t}-F_{s}|^{2}]\leq 2[E(|G_{t}-G_{s}|^2)+2E(|Z_{t}-Z_{s}|^{2}).
\end{equation}
From lemma1.2, we have $\mathbb{E}(G_{t}-G_{s})^{2}\leq C_{\beta}|t-s|^{2\beta}$. Furthermore, we have
\begin{equation}\label{317}
\begin{aligned}
E(|Z_{t}-Z_{s}|^{2})&=E[\int_{0}^{t}e^{-k(t-u)}dG_{u}-\int_{0}^{s}e^{-k(s-v)}dG_{v}]^{2}\\
&=E[e^{-kt}\int_{s}^{t}e^{ku}dG_{u}+(e^{-kt}-e^{-ks})\int_{0}^{s}e^{kv}dG_{v}]^{2}\\
&\leq E[e^{-kt}\int_{s}^{t}e^{ku}dG_{u}]^{2}+E[(e^{-k(t-s)}-1)e^{-ks}\int_{0}^{s}e^{kv}dG_{v}]^{2}.
\end{aligned}
\end{equation}
For the second term in (\ref{317}), we have
\begin{equation}
E[(e^{-k(t-s)}-1)e^{-ks}\int_{0}^{s}e^{kv}dG_{v}]^{2}\leq C|t-s|^{2\beta}.
\end{equation}
Meanwhile, we have
\begin{equation}
\begin{aligned}
E[e^{-kt}\int_{s}^{t}e^{ku)}dG_{u}]^{2}&\leq \int_{[s,t]^{2}}e^{-k(t-u)-k(t-v)}|u-v|^{2\beta-2}dudv+(\int_{s}^{t}e^{-k(t-u)u^{\beta-1}}du)^{2}\\
&\leq C^{'}|t-s|^{2\beta}.
\end{aligned}
\end{equation}
Hence, we obatin the desired result.
\end{proof}

\begin{proposition}
Under the hypothesis1.1, and $\gamma>\beta$ we can obtain that
$\underset{T\rightarrow \infty }{\lim }\frac{F_{T}}{T^\gamma}=0$ almost surely.
\end{proposition}

\begin{proof}
The proof is similar as \cite{Yong2017}.
When $\beta \in (\frac{1}{2},1)$, Chebyshev's inequality, the hypercontractivity of multiple Wiener-It$\hat{o}$ integrals imply that for any $\varepsilon >0$ and $p(\gamma-\beta )>-1$,
\begin{equation}
\begin{aligned}
p(\frac{F_{n}}{n^\gamma}>\varepsilon )\leq \frac{EF_{n}^{p}}{n^{\gamma p}\varepsilon ^{p}}%
\leq \frac{C(EF_{n}^{2})^{p/2}}{n^{\gamma p}\varepsilon ^{p}}\leq \frac{C}{%
n^{p(\gamma-\beta )}}.
\end{aligned}
\end{equation}
The Borel-Cantelli lemma implies for $\beta \in (\frac{1}{2},1)$,\begin{align}
 \lim_{n\rightarrow \infty }\frac{F_{n}}{n^{\gamma}}=0,a.s..
\end{align}
Second, there exist two constants $\alpha \in (0,1)$, $C_{\alpha ,\beta }>0$ independent of T such that any $|t-s|\leq 1$,\begin{align}
E[|F_{t}-F_{s}|^{2}]\leq C_{\alpha ,\beta }|t-s|^{2\beta }.
\end{align}
Then the Garsia-Rumsey inequality implies that for any real number $p>\frac{4}{\alpha }$, $q>1$, and integer
$n\geq 1$, \begin{align}
|F_{t}-F_{s}|\leq R_{p,q}n^{q/p},  \forall t,s\in[n,n+1],
\end{align}
where $R_{p,q}$ is a random constant independent of n.
Finally, since $|\frac{F_{T}}{T^\gamma}|\leq \frac{1}{T^\gamma}|F_{T}-F_{n}|+\frac{n^{\gamma}}{T^\gamma}\frac{|F_{n}|}{n^{\gamma}}$.
where $n=[T]$ is the biggest integer less than or equal to a real number T, we have $\frac{F_{T}}{T^\gamma} $ converges to 0 almost surely as $T\rightarrow \infty $.
\end{proof}

\begin{proposition}
Let $X_{T}$ be given by (\ref{23}), then
\begin{equation}\label{325}
\frac{1}{T}\int_{0}^{T}X_{t}^{2}dt%
{\rightarrow }C_{\beta }\Gamma (2\beta -1)k^{-2\beta }+\mu^2,
\end{equation}
almost surely, as T tends to infinity.
\end{proposition}

\begin{proof}
From the expression of $X_{t}$ in (\ref{23}), we obtain
\begin{equation}\label{326}
\begin{aligned}
\frac{1}{T}\int_{0}^{T}X_{t}^{2}dt&=\frac{1}{T}\int_{0}^{T}[\mu
(1-e^{-kt})+\int_{0}^{t}e^{-k(t-s)}dG_{s}]^2dt\\
&=\frac{1}{T}\int_{0}^{T}[\mu (1-e^{-kt})]^{2}dt+\frac{1}{T}%
\int_{0}^{T}[\int_{0}^{t}e^{-k(t-s)}dG_{s}]^{2}dt+\frac{2}{T}%
\int_{0}^{T}[\mu (1-e^{-kt})\int_{0}^{t}e^{-k(t-s)}dG_{s}]dt\\&=:I_{1}+I_{2}+I_{3}.\\
\end{aligned}
\end{equation}
For the term $I_{1}$, it is easy to see that
\begin{equation}\label{327}
I_{1}=\frac{1}{T}\int_{0}^{T}[\mu (1-e^{-kt})]^{2}dt\overset{a.s}{\rightarrow }\mu ^{2}.
\end{equation}
Using an argument similar to that in (3.4) of \cite{Yong2020}, we have
\begin{equation}\label{328}
I_{2}=\frac{1}{T}\int_{0}^{T}[\int_{0}^{t}e^{-k(t-s)}dG_{s}]^{2}dt\overset{a.s}%
{\rightarrow }C_{\beta }\Gamma (2\beta -1)k^{-2\beta }.
\end{equation}
We can deduce that
\begin{equation}\label{329}
I_{3}=\frac{2\mu }{T}\int_{0}^{T}(\int_{0}^{t}e^{-k(t-s)}dG_{s})dt-\frac{2\mu
}{T}\int_{0}^{T}e^{-kt}(\int_{0}^{t}e^{-k(t-s)}dG_{s})dt,
\end{equation}
a standard calculation yields
\begin{equation}\label{330}
\begin{aligned}
\frac{2\mu }{T}\int_{0}^{T}e^{-kt}(\int_{0}^{t}e^{-k(t-s)}dG_{s})dt
&=\frac{2\mu }{T}\int_{0}^{T}dG_{s}\int_{s}^{T}e^{-k(2t-s)}dt\\
&=\frac{2\mu }{T}\int_{0}^{T}\frac{1}{2k}(e^{-ks}-e^{-k(2T-s)})dG_{s}\\
&=\frac{\mu }{T}\int_{0}^{T}\frac{1}{k}e^{-ks}dG_{s}
-\frac{\mu }{T}\int_{0}^{T}\frac{1}{k}e^{-k(2T-s)}dG_{s}.\\
\end{aligned}
\end{equation}
For the first term in (\ref{330}), set $\frac{1}{T}\int_{0}^{T}e^{-ks}dG_{s}=M_{T}$, we have
\begin{equation}
\begin{aligned}
E[M_{T}^{2}]&=\int_{[0,T]^{2}}e^{-ks}\cdot e^{-kr}\cdot \frac{\partial
^{2}R(s,r)}{\partial s\partial r}dsdr\\
&\leq \int_{[0,T]^{2}}e^{-ks}\cdot e^{-kr}|s-r|^{2\beta
-2}dsdr+(\int_{0}^{T}e^{-ks}s^{\beta -1}ds)^{2}\leq C_{\beta }T^{2\beta }.\\
\end{aligned}
\end{equation}
When $\beta \in (\frac{1}{2},\frac{3}{4})$ and $p(\beta -1)<-1$,
\begin{equation}
\begin{aligned}
P(|\frac{M_{n}}{n}|>\varepsilon )&\leq \frac{E|M_{n}|^{p}}{n^{p}\varepsilon^{p}}
&\leq \frac{E(|M_{n}^{2}|)^{\frac{p}{2}}}{n^{p}\varepsilon ^{p}}
&\leq n^{p(\beta -1)},
\end{aligned}
\end{equation}
we can obtain $\underset{n\rightarrow \infty }{\lim }\frac{M_{n}}{n}=0,a.s.$.\\
Second, there exist a constant $C_{\beta }>0$ independent of T such that any $|t-s|\leq 1$,
\begin{equation}
  \begin{split}
    E[|M_{t}-M_{s}|^{2}]&=\int_{[s,t]^2}e^{-ks}e^{-kr}\frac{\partial^2 R(m,m)}{\partial m,\partial n}dmdn\\
    &\leq \int_{[s,t]^2}e^{-ks}e^{-kr}|m-n|^{2\beta-2}dmdn+(\int_{[s,t]}e^{-ks}e^{-kr}|mn|^{\beta-1}dmdn)^2\\
    &\leq E[|G_{t}-G_{s}|^{2}]\leq C_{\beta }(t-s)^{2\beta }.
  \end{split}
\end{equation}
Then the Garsia-Rumsey inequality implies that for any real number $p>\frac{4}{\alpha }$, $q>1$, and integer
$n\geq 1$, \begin{align}
|M_{t}-M_{s}|\leq R_{p,q}n^{q/p},  \forall t,s\in[n,n+1],
\end{align}
where $R_{p,q}$ is a random constant independent of n.
Finally, since $|\frac{M_{T}}{T}|\leq \frac{1}{T}|M_{T}-M_{n}|+\frac{n}{T}\frac{|M_{n}|}{n}$.
where $n=[T]$ is the biggest integer less than or equal to a real number T,we have $\frac{M_{T}}{T} $ converges to 0 almost surely as $T\rightarrow \infty $, we imply that
\begin{equation}\label{331}
\frac{\mu }{T}\int_{0}^{T}\frac{1}{k}e^{-ks}dG_{s}\overset{a.s}{\rightarrow }0.
\end{equation}
For the last term in (\ref{330}), we obtain
\begin{equation}\label{332}
\frac{\mu }{T}\int_{0}^{T}\frac{1}{k}e^{-k(2T-s)}dG_{s}=e^{-kT}\frac{\mu }{T}\int_{0}^{T}\frac{1}{k}e^{-k(T-s)}dG_{s}\overset{a.s}{\rightarrow }0,
\end{equation}
where the last step is from\cite{Yong2020}.\\
Combining the above result, we obtain
\begin{equation}\label{333}
\frac{2\mu }{T}\int_{0}^{T}e^{-kt}(\int_{0}^{t}e^{-k(t-s)}dG_{s})dt\overset{a.s}{\rightarrow }0.
\end{equation}
This implies that
\begin{equation}\label{334}
I_{3}=\frac{2}{T}\int_{0}^{T}[\mu (1-e^{-kt})\int_{0}^{t}e^{-k(t-s)}dG_{s}]dt\overset{a.s}{\rightarrow }0.
\end{equation}
By (\ref{326})-(\ref{334}), as T tends to infinity, we imply that
\begin{equation}
(\frac{1}{T}\int_{0}^{T}X_{t}^{2}dt)\overset{a.s}{\rightarrow }C_{\beta }\Gamma (2\beta -1)k^{-2\beta }+\mu^{2}.
\end{equation}
Hence, the second moment estimator $\widehat{k}$ is strongly consistent, namely, $\widehat{k}\overset{a.s}{\rightarrow }k$.
\end{proof}

\subsection{The least squares estimator}
\begin{proposition}\label{key}
Let $(X_{t},t\in[0,T])$ be given by (\ref{23}), then
\begin{equation}\label{336}
\frac{1}{T}\int_{0}^{T}X_{t}dX_{t}\overset{a.s}{\rightarrow } C_{\gamma},
\end{equation}
as T tends to infinity, where $C_{\gamma}$ denotes a suitable positive constant.
\end{proposition}

\begin{proof}
By (\ref{model}), we represent the stochastic integral $\frac{1}{T} X_tdX_t$ as
\begin{equation}
\frac{1}{T}\int_{0}^{T}X_{t}dX_{t}=\frac{k\mu}{T}\int_{0}^{T}X_{t}dt-\frac{k}{T}\int_{0}^{T}X_{t}^{2}dt+\frac{1}{T}\int_{0}^{T}X_{t}dG_{t}.
\end{equation}
By (\ref{23}), we have that
\begin{equation}\label{338}
\frac{1}{T}\int_{0}^{T}X_{t}dG_{t}=\frac{1}{T}\int_{0}^{T}\mu(1-e^{-kt})dG_{t}+\frac{1}{T}\int_{0}^{T}\int_{0}^{T}e^{-k(t-s)}dG_{s}dG_{t}.
\end{equation}
For the first term in (\ref{338}), by (\ref{331}), we have
\begin{equation}\label{339}
\frac{1}{T}\int_{0}^{T}\mu(1-e^{-kt})dG_{t}=\frac{\mu}{T}\int_{0}^{T}dG_{t}-\frac{1}{T}\int_{0}^{T}e^{-kt}dG_{t}\overset{a.s}{\rightarrow }0.
\end{equation}
From proposition3.7 in \cite{Yong2020}, we know that $\frac{1}{T}\int_{0}^{T}\int_{0}^{T}e^{-k(t-s)}dG_{s}dG_{t}$ converges almost surely, as T tends to infinity to 0.
Then, combining (\ref{338}) and (\ref{339}), it's suffices to show that
\begin{equation}\label{340}
\frac{1}{T}\int_{0}^{T}X_{t}dG_{t}\overset{a.s}{\rightarrow }0.
\end{equation}
Meanwhile, by proposition3.6, we imply that
\begin{equation}\label{341}
\frac{k\mu}{T}\int_{0}^{T}X_{t}dt-\frac{k}{T}\int_{0}^{T}X_{t}^{2}dt\overset{a.s}{\rightarrow }C_{\gamma}.
\end{equation}
Therefore, we can obtain the desired result in (\ref{336}).
\end{proof}

\begin{proposition}
Let $\beta \in [\frac{1}{2},1)$, then we have $\hat{\mu} _{LS}\overset{a.s}{\rightarrow}\mu $.
\end{proposition}

\begin{proof}
Recall that,
\begin{equation}\label{342}
\widehat{\mu }_{LS}-\mu =\frac{\frac{X_{T}}{T}\cdot
\frac{1}{T}\int_{0}^{T}X_{t}^{2}dt-\frac{1}{T}\int_{0}^{T}X_{t}dX_{t}\cdot
\frac{1}{T}\int_{0}^{T}X_{t}dt-\mu (\frac{X_{T}}{T}\cdot
\frac{1}{T}\int_{0}^{T}X_{t}dt-\frac{1}{T}\int_{0}^{T}X_{t}dX_{t})}{%
\frac{X_{T}}{T}\cdot \frac{1}{T}\int_{0}^{T}X_{t}dt-\frac{1}{T}%
\int_{0}^{T}X_{t}dX_{t}}.
\end{equation}
Denote partI as follow, and according to proposition\ref{key}, we have
\begin{equation}\label{343}
partI=\frac{1}{T}\int_{0}^{T}X_{t}dX_{t}[\frac{1}{T}\int_{0}^{T}X_{t}dt-\mu ]\overset{a.s}{\rightarrow }0.
\end{equation}
Denote partII as follow, by proposition3.6, it's easy to see that
\begin{equation}\label{344}
partII=\frac{X_{T}}{T}(\frac{1}{T}\int_{0}^{T}X_{t}^{2}dt-\frac{\mu}{T}\int_{0}^{T}X_{t}dt)\overset{a.s}{\rightarrow }0.
\end{equation}
Finally, we obtain the desired result.
\end{proof}

\begin{proposition}
Let $\beta\in[\frac{1}{2},1)$, then we have $\hat{k}_{LS}\overset{a.s}{\rightarrow }k$.
\end{proposition}

\begin{proof}
Since the expression for $\hat{k}_{LS}$ in (\ref{109}), we write
\begin{equation}\label{345}
\widehat{k}_{LS}-k
=\frac{\frac{1}{T^{2}}X_{T}\int_{0}^{T}X_{t}dt-k\mu \frac{1}{T}%
\int_{0}^{T}X_{t}dt-\frac{1}{T}\int_{0}^{T}X_{t}dG_{t}+k(\frac{1}{T}%
\int_{0}^{T}X_{t}dt)^{2}}{\frac{1}{T}\int_{0}^{T}X_{t}^{2}dt-(\frac{1}{T}%
\int_{0}^{T}X_{t}dt)^{2}}.
\end{equation}
First, denote partA as follow,
\begin{equation}\label{346}
partA=k(\frac{1}{T}\int_{0}^{T}X_{t}dt)^{2}-k\mu \frac{1}{T}\int_{0}^{T}X_{t}dt\overset{a.s}{\rightarrow }0.
\end{equation}
Then, denote part B as follow,
\begin{equation}\label{34.}
\begin{aligned}
partB=\frac{1}{T^{2}}X_{T}\int_{0}^{T}X_{t}dt&=\frac{X_{T}}{T}\frac{1}{T}\int_{0}^{T}X_{t}dt.
\\
&=[\frac{1}{T}\mu(1-e^{-kT})+\frac{1}{T}\int_{0}^{T}e^{-k(t-s)}dG_{s}]\cdot\frac{1}{T}\int_{0}^{T}X_{t}dt.
\end{aligned}
\end{equation}
From proposition 3.7 in \cite{Yong2020} and (\ref{24}), we obtain
\begin{equation}\label{347}
\frac{1}{T^{2}}X_{T}\int_{0}^{T}X_{t}dt\overset{a.s}{\rightarrow }0.
\end{equation}
Combining (\ref{340}), (\ref{346}) and (\ref{347}), we proves the claim of the Proposition.
\end{proof}

\section{Asymptotic behaviors for $k > 0$}

\subsection{The moment estimator}

We need several lemmas, providing sufficient conditions to prove the asymptotic normality of $\widehat{\mu}$.
\begin{lemma}
When $\beta \in (\frac{1}{2},1)$,
\begin{equation}\label{400}
T^{\beta-\frac{1}{2}}(\frac{1}{T}\int_{0}^{T}X_{t}dt-\mu )\to 0,
\end{equation}
almost surely as $T\to \infty$.
\end{lemma}

\begin{proof}
\begin{equation}
\begin{aligned}
T^{\beta -\frac{1}{2}}(\frac{1}{T}\int_{0}^{T}X_{t}dt-\mu )
&=T^{\beta -\frac{1}{2}}(\frac{1}{T}\int_{0}^{T}\mu (1-e^{-kt})dt+\frac{1}{T}%
\int_{0}^{T}\int_{0}^{t}e^{-k(t-s)}dG_{s}dt-\mu )\\
&=T^{\beta -\frac{1}{2}}(\frac{1}{T}\int_{0}^{T}\int_{0}^{t}e^{-k(t-s)}dG_{s}dt)\\
&=\frac{1}{T^{\frac{3}{2}-\beta }}\int_{0}^{T}\int_{0}^{t}e^{-k(t-s)}dG_{s}dt\\
&=\frac{k}{T^{\frac{3}{2}-\beta }}\int_{0}^{T}(1-e^{-k(T-s)})dG_{s}\\
&=\frac{k}{T^{\frac{3}{2}-\beta }}\cdot G_{T}-\frac{k\left\Vert f_{T}\right\Vert}{T^{\frac{3}{%
2}-\beta }}\frac{I_{1}(f_{T})}{\left\Vert f_{T}\right\Vert}.\\
\end{aligned}
\end{equation}
where $\frac{3}{2}-\beta >\beta $,$\frac{G_{T}}{T^{\beta }}$ and $\frac{I_{1}(f_{T})}{||f_{T}||}$ are also the stanard noraml distribution of random variables, which together with Proposition3.5  proves the claim of the lemma.
\end{proof}

\begin{proposition}
For $\beta\in[\frac{1}{2},1)$, we obtain for $\hat\mu$ defined by (\ref{4}).
\begin{equation}\label{402}
T^{1-\beta }(\hat{\mu }-\mu )\overset{law}{\to }N(0,\frac{1}{k^{2}}).
\end{equation}
\end{proposition}

\begin{proof}
First, we have
\begin{equation}\label{403}
\begin{aligned}
T^{1-\beta }(\widehat{\mu}-\mu )
&=T^{1-\beta }[\frac{1}{T}\int_{0}^{T}\mu (1-e^{-kt})dt+\frac{1}{T}%
\int_{0}^{T}\int_{0}^{t}e^{-k(t-s)})dG_{s}dt-\mu ]\\
&=\frac{1}{T^{\beta }}\int_{0}^{T}\int_{0}^{t}e^{-k(t-s)})dG_{s}dt\\
&=\frac{G_{T}}{kT^{\beta }}-\frac{1}{kT^{\beta }}\int_{0}^{T}e^{-k(T-s)})dG_{s}.\\
\end{aligned}
\end{equation}
From Proposition3.8 in \cite{Yong2020}, we have $ \frac{1}{kT^{\beta }}\int_{0}^{T}e^{-k(T-s)})dG_{s}\rightarrow 0$, and
$\frac{G_{T}}{T^{\beta }}$ is a standard normal distribution. Finally, by the Slutsky's theorem we get the asymptotic normality (\ref{402}) holds.
\end{proof}

\begin{proposition}
Denote a constant that depends on $k$ and $\beta$ as $a:=C_{\beta}\Gamma(2\beta-1)k^{-2\beta}$,
then for $\beta\in[\frac{1}{2},\frac{3}{4})$ and $T\to\infty$, we have
\begin{equation}
\sqrt{T}(\hat{k}-k){\rightarrow}N(0,a^{2}\sigma_{\beta}^{2}/4\beta^{2}),
\end{equation}
where $\sigma _{\beta }^{2}=(4\beta -1)[1+\frac{\Gamma (3-4\beta )\Gamma (4\beta-1)}{\Gamma (2\beta )\Gamma (2-2\beta )}]$.
\end{proposition}
\begin{proof}
First, we obtain
\begin{equation}
\begin{aligned}
\sqrt{T}(\frac{1}{T}\int_{0}^{T}X_{t}^{2}dt-(\frac{1}{T}\int_{0}^{T}X_{t}dt)^{2}-a)
&=\sqrt{T}\biggl(\frac{1}{T}\int_{0}^{T}[\mu(1-e^{-kt})]^{2}dt
+\frac{1}{T}\int_{0}^{T}[\int_{0}^{t}e^{-k(t-s)}dG_{s}]^{2}dt\\
&+\frac{2}{T}\int_{0}^{T}[\mu(1-e^{-kt})\int_{0}^{t}e^{-k(t-s)}dG_{s}]dt-(\frac{1}{T}\int_{0}^{T}X_{t}dt)^{2}-a\biggr).\\
\end{aligned}
\end{equation}
In fact, we have
\begin{equation}
\frac{1}{T}\int_{0}^{T}[\mu(1-e^{-kt})]^{2}dt-(\frac{1}{T}\int_{0}^{T}X_{t}dt)^{2}\overset{a.s}\to 0.
\end{equation}
Meanwhile, by (\ref{331}), we can write
\begin{equation}
\frac{\mu }{\sqrt{T}}\int_{0}^{T}\frac{1}{k}e^{-ks}dG_{s}\overset{a.s}\to 0.
\end{equation}
From Proposition3.8 in \cite{Yong2020}, we have
\begin{equation}
\frac{\mu }{\sqrt{T}}\int_{0}^{T}\frac{1}{k}e^{-k(2T-s)}dG_{s}\overset{a.s}\to 0.
\end{equation}
Those two facts now imply that,
\begin{equation}
\sqrt{T}(\frac{1}{T}\int_{0}^{T}X_{t}^{2}dt-(\frac{1}{T}%
\int_{0}^{T}X_{t}dt)^{2}-a)\overset{law}{\rightarrow }N(0,a^{2}\sigma_{\beta }^{2}/k).
\end{equation}
Finally, since the delta method implies that the asymptotic normality holds.
\end{proof}

\subsection{The least squares estimator}
Let us now disscus the asymptotic normality of LSE $\hat{\mu}_{LS}$ and $\hat{k}_{LS}$.
\begin{proposition}
For $\beta \in [\frac{1}{2},1)$, and $T\to \infty$,
\begin{equation}
T^{1-\beta }(\widehat{\mu}_{LS}-\mu )\overset{law}{\to }N(0,\frac{1}{k^{2}}).
\end{equation}
\end{proposition}
\begin{proof}
By the representation of (\ref{342}),
\begin{equation}
T^{1-\beta }(\widehat{\mu }_{LS}-\mu )=\frac{T^{1-\beta}(partI + partII)}{%
\frac{X_{T}}{T}\cdot \frac{1}{T}\int_{0}^{T}X_{t}dt-\frac{1}{T}%
\int_{0}^{T}X_{t}dX_{t}}.
\end{equation}
First, Combining (\ref{336}) and  Proposition4.1, $T^{1-\beta }\cdot partI$ can write as,
\begin{equation}
T^{1-\beta }\cdot \frac{1}{T}\int_{0}^{T}X_{t}dX_{t}[\frac{1}{T}%
\int_{0}^{T}X_{t}dt-\mu ]\overset{law}{\to }N(0,\frac{1}{k^{2}}).
\end{equation}
Using arguments similar to strong convergence of $\widehat{\mu}$, we can easily obtain
\begin{equation}
T^{1-\beta }\cdot partII \overset{a.s}\to 0.
\end{equation}
Then, applying Slutsky's theorem, we obtain the desired result.
\end{proof}

\begin{proposition}
when $k>0$, and $\beta \in (\frac{1}{2},\frac{3}{4})$, then the following convergence results hold true
\begin{equation}
\sqrt{T}(\widehat{k}_{LS}-k)\overset{law}{\rightarrow }N(0,4ka^{2}\sigma _{\beta }^{2}).
\end{equation}

\end{proposition}
\begin{proof}
  From (\ref{345}), we have
  \begin{equation}
  \sqrt{T}(\widehat{k}_{LS}-k)
  =\frac{\biggl(\sqrt{T}\frac{1}{T^{2}}X_{T}\int_{0}^{T}X_{t}dt-k\mu \frac{1}{T}%
  \int_{0}^{T}X_{t}dt-\frac{1}{T}\int_{0}^{T}X_{t}dG_{t}+k(\frac{1}{T}%
  \int_{0}^{T}X_{t}dt)^{2}\biggr)}{\frac{1}{T}\int_{0}^{T}X_{t}^{2}dt-(\frac{1}{T}%
  \int_{0}^{T}X_{t}dt)^{2}}.
  \end{equation}
  where we first consider only two terms of it, namely, the $\sqrt{T}\cdot partA$ can write as
  \begin{equation}
  \begin{aligned}
  \sqrt{T}[k(\frac{1}{T}\int_{0}^{T}X_{t}dt)^{2}-\frac{k\mu }{T}%
  \int_{0}^{T}X_{t}dt]&=\sqrt{T}[(\frac{k}{T}\int_{0}^{T}X_{t}dt-k\mu )\frac{1}{%
  T}\int_{0}^{T}X_{t}dt]\\
  &=\sqrt{T}[(\frac{k}{T}\int_{0}^{T}\mu (1-e^{-kt})dt+\frac{k}{T}%
  \int_{0}^{T}\int_{0}^{t}e^{-k(t-s)}dG_{s}dt-k\mu )\frac{1}{T}%
  \int_{0}^{T}X_{t}dt]\\
  &=[\frac{k\mu }{\sqrt{T}}\int_{0}^{T}-e^{-kt}dt+\frac{k}{\sqrt{T}}%
  \int_{0}^{T}\int_{0}^{t}e^{-k(t-s)}dG_{s}dt]\frac{1}{T}\int_{0}^{T}X_{t}dt.\\
  \end{aligned}
  \end{equation}
  where $\frac{k\mu }{\sqrt{T}}\int_{0}^{T}-e^{-kt}dt\overset{a.s}{\rightarrow }0$, we can imply that
  \begin{equation}
  \begin{aligned}
  \frac{k}{\sqrt{T}}\int_{0}^{T}\int_{0}^{t}e^{-k(t-s)}dG_{s}dt&=\frac{k}{%
  \sqrt{T}}\int_{0}^{T}dG_{s}\int_{s}^{T}e^{-k(t-s)}dt=\frac{1}{\sqrt{T}}%
  \int_{0}^{T}dG_{s}-\frac{1}{\sqrt{T}}\int_{0}^{T}e^{-k(T-s)}dG_{s}\\
  &=\frac{G_{T}}{\sqrt{T}}-\frac{1}{\sqrt{T}}e^{-kT}\int_{0}^{T}e^{ks}dG_{s}.\\
  \end{aligned}
  \end{equation}
  Using the Proposition3.1, we can obtain that $\frac{1}{\sqrt{T}}e^{-kT}\int_{0}^{T}e^{ks}dG_{s}\overset{a.s}{\rightarrow }0$, we deduce that
  \begin{equation}
  \begin{aligned}
  \sqrt{T}(\widehat{k}_{LS}-k)&=-\frac{\frac{1}{\sqrt{T}}%
  \int_{0}^{T}X_{t}dG_{t}+(\frac{X_{T}}{\sqrt{T}}+\frac{G_{T}}{\sqrt{T}})\frac{%
  1}{T}\int_{0}^{T}X_{t}dt}{\frac{1}{T}\int_{0}^{T}X_{t}^{2}dt-(\frac{1}{T}%
  \int_{0}^{T}X_{t}dt)^{2}}\\
  &=-\frac{\frac{1}{\sqrt{T}}\int_{0}^{T}\mu (1-e^{-kt})dG_{t}+\frac{1}{\sqrt{T%
  }}\int_{0}^{T}\int_{0}^{t}e^{-k(t-s)}dG_{t}dG_{s}}{\frac{1}{T}%
  \int_{0}^{T}X_{t}^{2}dt-(\frac{1}{T}\int_{0}^{T}X_{t}dt)^{2}}+\frac{(\frac{%
  X_{T}}{\sqrt{T}}+\frac{G_{T}}{\sqrt{T}})\frac{1}{T}\int_{0}^{T}X_{t}dt}{%
  \frac{1}{T}\int_{0}^{T}X_{t}^{2}dt-(\frac{1}{T}\int_{0}^{T}X_{t}dt)^{2}}\\
  &=-\frac{\mu }{\sqrt{T}}G_{T}+\frac{\mu }{\sqrt{T}}\int_{0}^{T}e^{-kt}dG_{t}-%
  \frac{1}{\sqrt{T}}\int_{0}^{T}\int_{0}^{t}e^{-k(t-s)}dG_{t}dG_{s}+\frac{G_{T}%
  }{\sqrt{T}}\frac{1}{T}\int_{0}^{T}X_{t}dt\\
  &=\frac{G_{T}}{\sqrt{T}}(\frac{1}{T}\int_{0}^{T}X_{t}dt-\mu )+\frac{\mu }{%
  \sqrt{T}}\int_{0}^{T}e^{-kt}dG_{t}-\frac{1}{\sqrt{T}}\int_{0}^{T}%
  \int_{0}^{t}e^{-k(t-s)}dG_{t}dG_{s}.\\
  \end{aligned}
  \end{equation}
  It's also clear that $\frac{\mu }{\sqrt{T}}\int_{0}^{T}e^{-kt}dG_{t}\overset{a.s}{\rightarrow }0$, Using the lemma4.1, we can imply that$\frac{G_{T}}{\sqrt{T}}(\frac{1}{T}\int_{0}^{T}X_{t}dt-\mu
  )\overset{a.s}{\rightarrow }0$. Hence we have
  \begin{equation}
  \begin{aligned}
  \sqrt{T}(\widehat{k}_{LS}-k)&=\frac{-\frac{1}{\sqrt{T}}\int_{0}^{T}%
  \int_{0}^{t}e^{-k(t-s)}dG_{t}dG_{s}}{\frac{1}{T}\int_{0}^{T}X_{t}^{2}dt-(%
  \frac{1}{T}\int_{0}^{T}X_{t}dt)^{2}}&=\frac{-\frac{1}{\sqrt{T}}%
  I_{2}(e^{-k(t-\cdot )})}{\frac{1}{T}\int_{0}^{T}X_{t}^{2}dt-(\frac{1}{T}%
  \int_{0}^{T}X_{t}dt)^{2}},\\
  \end{aligned}
  \end{equation}
  where $\sigma _{\beta }^{2}=(4\beta -1)[1+\frac{\Gamma (3-4\beta )\Gamma (4\beta-1)}{\Gamma (2\beta )\Gamma (2-2\beta )}]$. \par
  From (4.8) in \cite{Yong2020}, we know that $-\frac{1}{\sqrt{T}}%
  I_{2}(e^{-k(t-\cdot )}) \overset{law}\to N(0,4ka^2\sigma_{\beta}^{2})$. Thus, combining with (\ref{325}), the Slutsky's theorem implies that the asymptotic normality holds.\smartqed
  \end{proof}
\begin{acknowledgements}
  This research is partly supported by NSFC(No.11961033).
\end{acknowledgements}

\nocite{*}
\bibliographystyle{abbrvnat}
\bibliography{thesis}

\begin{thebibliography}{15}
\providecommand{\natexlab}[1]{#1}
\providecommand{\url}[1]{\texttt{#1}}
\expandafter\ifx\csname urlstyle\endcsname\relax
  \providecommand{\doi}[1]{doi: #1}\else
  \providecommand{\doi}{doi: \begingroup \urlstyle{rm}\Url}\fi

\bibitem[Chen and Zhou(2020)]{Yong2020}
Y.~Chen and H.~Zhou.
\newblock Parameter estimation for an ornstein-uhlenbeck process driven by a
  general gaussian noise.
\newblock \emph{Acta Mathematica Scientia(2020+), accepted, arXiv:2002.09641},
  2020.

\bibitem[Chen et~al.(2017)Chen, Hu, and Wang]{Yong2017}
Y.~Chen, Y.~Hu, and Z.~Wang.
\newblock Parameter estimation of complex fractional ornstein-uhlenbeck
  processes with fractional noise.
\newblock \emph{ALEA}, 14\penalty0 (1):\penalty0 613--629, 2017.

\bibitem[Fergusson and Platen(2015)]{Fergusson2015}
K.~Fergusson and E.~Platen.
\newblock Application of maximum likelihood estimation to stochastic short rate
  models.
\newblock \emph{Annals of Financial Economics}, 10\penalty0 (02):\penalty0
  1550009, 2015.

\bibitem[Hu and Nualart(2010)]{Hu2010}
Y.~Hu and D.~Nualart.
\newblock Parameter estimation for fractional ornstein--uhlenbeck processes.
\newblock \emph{Statistics probability letters}, 80\penalty0 (11-12):\penalty0
  1030--1038, 2010.

\bibitem[Huang and Huang(2012)]{Jingzhi2012}
J.-Z. Huang and M.~Huang.
\newblock How much of the corporate-treasury yield spread is due to credit
  risk?
\newblock \emph{The Review of Asset Pricing Studies}, 2\penalty0 (2):\penalty0
  153--202, 2012.

\bibitem[Jolis(2007)]{Jolis2007}
M.~Jolis.
\newblock On the wiener integral with respect to the fractional brownian motion
  on an interval.
\newblock \emph{Journal of mathematical analysis and applications},
  330\penalty0 (2):\penalty0 1115--1127, 2007.

\bibitem[Nourdin and Tran(2019)]{Nourdin2018}
I.~Nourdin and T.~D. Tran.
\newblock Statistical inference for vasicek-type model driven by hermite
  processes.
\newblock \emph{Stochastic Processes and their Applications}, 129\penalty0
  (10):\penalty0 3774--3791, 2019.

\bibitem[Nualart(2006)]{Nourdin2006}
D.~Nualart.
\newblock \emph{The Malliavin calculus and related topics}, volume 1995.
\newblock Springer, 2006.

\bibitem[Nualart et~al.(2005)Nualart, Peccati, et~al.]{Nualart2005}
D.~Nualart, G.~Peccati, et~al.
\newblock Central limit theorems for sequences of multiple stochastic
  integrals.
\newblock \emph{The Annals of Probability}, 33\penalty0 (1):\penalty0 177--193,
  2005.

\bibitem[Vasicek(1977)]{Vasicek1977}
O.~Vasicek.
\newblock An equilibrium characterization of the term structure.
\newblock \emph{Journal of financial economics}, 5\penalty0 (2):\penalty0
  177--188, 1977.

\bibitem[Wu et~al.(2020)Wu, Dong, Lv, and Wang]{San2020}
S.~Wu, Y.~Dong, W.~Lv, and G.~Wang.
\newblock Optimal asset allocation for participating contracts with mortality
  risk under minimum guarantee.
\newblock \emph{Communications in Statistics-Theory and Methods}, 49\penalty0
  (14):\penalty0 3481--3497, 2020.

\bibitem[Xiao and Yu(2017)]{Xiao2017}
W.~Xiao and J.~Yu.
\newblock Asymptotic theory for estimating drift parameters in the fractional
  vasicek model.
\newblock \emph{Econometric Theory}, 35:\penalty0 198--231, 2017.

\bibitem[Xiao et~al.(2018)Xiao, Zhang, and Zuo]{Xiao2018}
W.~Xiao, X.~Zhang, and Y.~Zuo.
\newblock Least squares estimation for the drift parameters in the
  sub-fractional vasicek processes.
\newblock \emph{Journal of Statistical Planning and Inference}, 197:\penalty0
  141--155, 2018.

\bibitem[Yang(2013)]{yang2013}
B.~H. Yang.
\newblock Estimating long-run pd, asset correlation, and portfolio level pd by
  vasicek models.
\newblock \emph{Journal of Risk Model Validation}, 7\penalty0 (4), 2013.

\bibitem[Yu(2020)]{Yu2018}
Q.~Yu.
\newblock Statistical inference for vasicek-type model driven by self-similar
  gaussian processes.
\newblock \emph{Communications in Statistics-Theory and Methods}, 49\penalty0
  (2):\penalty0 471--484, 2020.

\end{thebibliography}
\end{document}